\documentclass[reqno]{amsart}
\usepackage{amssymb}
\usepackage{graphicx}

\usepackage[usenames, dvipsnames]{color}
\usepackage{verbatim}
\usepackage{mathrsfs}
\usepackage{bm}
\usepackage{cite}

\numberwithin{equation}{section}

\newtheorem{theorem}{Theorem}[section]

\newtheorem{example}[theorem]{Example}

\theoremstyle{definition}
\newtheorem{remark}[theorem]{Remark}

\theoremstyle{definition}

\theoremstyle{definition}

\makeatletter
\def\dashint{\operatorname%
{\,\,\text{\bf-}\kern-.98em\DOTSI\intop\ilimits@\!\!}}
\makeatother

\def\\det{\text{det}}

\def\.5{\frac{1}{2}}

\newcommand{\RN}[1]{%
  \textup{\uppercase\expandafter{\romannumeral#1}}%
}

\renewcommand{\epsilon}{\varepsilon}

\newcounter{marnote}

\begin{document}
\title[Convergence for the fractional $p$-Laplacian]{Convergence for the fractional $p$-Laplacian and its corresponding extended Nirenberg problem}


\author[Z.W. Zhao]{Zhiwen Zhao}

\address[Z.W. Zhao]{Beijing Computational Science Research Center, Beijing 100193, China.}
\email{zwzhao365@163.com}


\date{\today} 



\begin{abstract}
The main objective of this paper is to establish the convergence for the fractional $p$-Laplacian of nonnegative sequence of functions with $p>2$. Further, we show the blow-up phenomena for solutions to the extended Nirenberg problem modeled by fractional $p$-Laplacian with the prescribed negative functions.
\end{abstract}

\maketitle



\section{Introduction and main results}

Let $n\geq1$, $p\geq2$ and $0<\sigma<1$. Define the fractional $p$-Laplacian $(-\Delta)^{\sigma}_{p}$ as follows:
\begin{align*}
(-\Delta)^{\sigma}_{p}u(x)=c_{n,\sigma p}P.V.\int_{\mathbb{R}^{n}}\frac{|u(x)-u(y)|^{p-2}(u(x)-u(y))}{|x-y|^{n+\sigma p}}dy,
\end{align*}
where $c_{n,\sigma p}$ is a positive constant and $P.V.$ represents the Cauchy principal value. It is worth pointing out that $(-\Delta)^{\sigma}_{p}$ becomes the linear fractional Laplacian operator $(-\Delta)^{\sigma}$ if $p=2$, while it is a nonlinear nonlocal operator if $p>2$. The definition of $(-\Delta)^{\sigma}_{p}u$ is valid under the condition that $u\in C^{\sigma p+\alpha}_{loc}(\mathbb{R}^{n})\cap\mathcal{L}_{\sigma p}(\mathbb{R}^{n})$ for some $\alpha>0$, where $C^{\sigma p+\alpha}_{loc}:=C^{[\sigma p+\alpha],\sigma p+\alpha-[\sigma p+\alpha]}_{loc}$ with $[\sigma p+\alpha]$ denoting the integer part of $\sigma p+\alpha$,
\begin{align*}
\mathcal{L}_{\sigma p}(\mathbb{R}^{n}):=\left\{u\in L^{p-1}_{loc}(\mathbb{R}^{n})\,\Big|\;\int_{\mathbb{R}^{n}}\frac{|u(x)|^{p-1}}{1+|x|^{n+\sigma p}}dx<\infty \right\}.
\end{align*}

Recently, Du, Jin, Xiong and Yang \cite{DJXY2021} derived the following fact:
\begin{align*}
``&\text{If $u_{i}\rightarrow u$ in $C^{2\sigma+\alpha}_{loc}$ as $i\rightarrow\infty$, and $\{(-\Delta)^{\sigma}u_{i}\}$ converges pointwisely,} \\ &\text{then $(-\Delta)^{\sigma}u_{i}\rightarrow (-\Delta)^{\sigma}u-\theta$ for some $\theta\geq0$,}"
\end{align*}
where the constant $\theta$ may be greater than zero, which is different from the classical Laplacian operator. However, the proof for the linear fractional Laplacian cannot be directly used to deal with the nonlinear case. So, in this paper we aim to overcome the nonlinear difficulty for the fractional $p$-Laplacian operator and prove that the above fact also holds for the nonlinear nonlocal operator $(-\Delta)^{\sigma}_{p}$ with $p>2$. The principal result of this paper is stated as follows.

\begin{theorem}\label{thm001}
Let $n\geq1$, $p>2$, $0<\sigma<1$ and $\alpha>0$. Assume that a sequence of nonnegative functions $\{u_{i}\}\subset \mathcal{L}_{\sigma p}(\mathbb{R}^{n})\cap C^{\sigma p+\alpha}_{loc}(\mathbb{R}^{n})$ converges in $C_{loc}^{\sigma p+\alpha}(\mathbb{R}^{n})$ to a function $u\in \mathcal{L}_{\sigma p}(\mathbb{R}^{n})$, and $\{(-\Delta)^{\sigma}_{p}u_{i}\}$ converges pointwisely in $\mathbb{R}^{n}$. Then for any $x\in\mathbb{R}^{n}$,
\begin{align*}
\lim_{i\rightarrow\infty}(-\Delta)^{\sigma}_{p}u_{i}(x)=(-\Delta)^{\sigma}_{p}u(x)-\theta,
\end{align*}
where $\theta$ is a nonnegative constant given by
\begin{align*}
\theta=c_{n,\sigma p}\lim_{R\rightarrow\infty}\lim_{i\rightarrow\infty}\int_{B_{R}^{c}}\frac{u_{i}^{p-1}(x)}{|x|^{n+\sigma p}}dx.
\end{align*}

\end{theorem}

\begin{proof}[Proof of Theorem \ref{thm001}]
For any fixed $x\in\mathbb{R}^{n}$ and $R>>|x|+1$, let
\begin{align}\label{MZ000}
&(-\Delta)^{\sigma}_{p}u(x)-(-\Delta)^{\sigma}_{p}u_{i}(x)\notag\\
=&c_{n,\sigma p}\int_{B_{R}(0)}\frac{|u(x)-u(y)|^{p-2}(u(x)-u(y))-|u_{i}(x)-u_{i}(y)|^{p-2}(u_{i}(x)-u_{i}(y))}{|x-y|^{n+\sigma p}}dy\notag\\
&+c_{n,\sigma p}\int_{B_{R}^{c}(0)}\frac{|u(x)-u(y)|^{p-2}(u(x)-u(y))}{|x-y|^{n+\sigma p}}dy\notag\\
&+c_{n,\sigma p}\int_{B_{R}^{c}(0)}\frac{-|u_{i}(x)-u_{i}(y)|^{p-2}(u_{i}(x)-u_{i}(y))}{|x-y|^{n+\sigma p}}dy\notag\\
=&\Phi_{i}(x,R)+\mathfrak{G}(x,R)+\Psi_{i}(x,R).
\end{align}
In light of the fact that $u_{i}\rightarrow u$ in $C^{\sigma p+\alpha}(B_{2R}(0))$, we obtain that for each $0<\varepsilon<1$, there exists an integer $N>0$ such that for every $i>N$,
\begin{align}\label{ZW001}
\|u_{i}-u\|_{C^{\sigma p+\alpha}(B_{2R}(0))}\leq \varepsilon^{\frac{p}{\min\{1,p-2\}}},\quad \|u_{i}\|_{C^{\sigma p+\alpha}(B_{2R}(0))}\leq\|u\|_{C^{\sigma p+\alpha}(B_{2R}(0))}+1.
\end{align}
Define
\begin{align*}
\Phi_{i}(x,R\setminus\varepsilon):=\Phi_{i}(x,R)-\Phi_{i}(x,\varepsilon),\quad\mathcal{M}:=\|u\|_{C^{\sigma p+\alpha}(B_{2R}(0))}+1,
\end{align*}
where $\Phi_{i}(x,\varepsilon)$ denotes the integral in $\Phi_{i}(x,R)$ with the domain $B_{R}(0)$ replaced by $B_{\varepsilon}(x)$. Using \eqref{ZW001}, we deduce that for $x,y\in B_{2R}(0)$, $i>N$,
\begin{align*}
&\left||u(x)-u(y)|^{p-2}(u(x)-u(y))-|u_{i}(x)-u_{i}(y)|^{p-2}(u_{i}(x)-u_{i}(y))\right|\notag\\
&\leq|u(x)-u(y)|^{p-2}|(u-u_{i})(x)-(u-u_{i})(y)|\notag\\
&\quad+\left||u(x)-u(y)|^{p-2}-|u_{i}(x)-u_{i}(y)|^{p-2}\right||u_{i}(x)-u_{i}(y)|\notag\\
&\leq C(p,\mathcal{M})\|u_{i}-u\|^{\min\{1,p-2\}}_{L^{\infty}(B_{2R}(0))}\leq C(p,\mathcal{M})\varepsilon^{p},
\end{align*}
which yields that
\begin{align}\label{ZW002}
\left|\Phi_{i}(x,R\setminus\varepsilon)\right|\leq& C(p,\mathcal{M})\varepsilon^{p}\int_{B_{2R}(x)\setminus B_{\varepsilon}(x)}\frac{dy}{|x-y|^{n+\sigma p}}\leq C(p,n,\sigma,\mathcal{M})\varepsilon^{(1-\sigma)p}.
\end{align}

On the other hand, if $\sigma p+\alpha\in(0,1]$, then it follows from \eqref{ZW001} that
\begin{align}\label{ZW003}
\left|\Phi_{i}(x,\varepsilon)\right|\leq& C(p,\sigma,\alpha,\mathcal{M})\int_{B_{\varepsilon}(x)}\frac{|x-y|^{(\sigma p+\alpha)(p-1)}}{|x-y|^{n+\sigma p}}\notag\\
\leq& C(p,n,\sigma,\alpha,\mathcal{M})\varepsilon^{(\sigma p+\alpha)(p-2)+\alpha}.
\end{align}
When $\sigma p+\alpha\in(1,\infty)$, utilizing \eqref{ZW001} again, it follows from Taylor expansion that
\begin{align*}
&\left||u_{i}(x)-u_{i}(y)|^{p-2}(u_{i}(x)-u_{i}(y))-|\nabla u_{i}(x)(x-y)|^{p-2}\nabla u_{i}(x)(x-y)\right|\notag\\
&\leq C(p,\sigma,\alpha,\mathcal{M})\left(|\nabla u_{i}(x)(x-y)|^{p-2}+|x-y|^{\min\{2,\sigma p+\alpha\}(p-2)}\right)|x-y|^{\min\{2,\sigma p+\alpha\}}\notag\\
&\leq C(p,\sigma,\alpha,\mathcal{M})|x-y|^{\min\{p,(\sigma+1)p+\alpha-2\}},
\end{align*}
where we utilized the following element inequality
\begin{align*}
\left||a|^{p-2}a-|b|^{p-2}b\right|\leq C(p)|a-b|\left(|a-b|^{p-2}+|b|^{p-2}\right),\quad\text{for $a,b\in\mathbb{R}^{n}$.}
\end{align*}
By the same argument, we have
\begin{align*}
&\left||u(x)-u(y)|^{p-2}(u(x)-u(y))-|\nabla u(x)(x-y)|^{p-2}\nabla u(x)(x-y)\right|\notag\\
&\leq C(p,\sigma,\alpha,\mathcal{M})|x-y|^{\min\{p,(\sigma+1)p+\alpha-2\}}.
\end{align*}
Therefore, we obtain that if $\sigma p+\alpha\in(1,\infty)$,
\begin{align}\label{ZW005}
\left|\Phi_{i}(x,\varepsilon)\right|\leq& C(p,\sigma,\alpha,\mathcal{M})\int_{B_{\varepsilon}(x)}\frac{|x-y|^{\min\{p,(\sigma+1)p+\alpha-2\}}}{|x-y|^{n+\sigma p}}dy\notag\\
\leq& C(p,n,\sigma,\alpha,\mathcal{M})\varepsilon^{\min\{(1-\sigma)p,p+\alpha-2\}},
\end{align}
where we utilized the anti-symmetry of $\nabla u(x)(x-y)$ and $\nabla u_{i}(x)(x-y)$ with regard to the center $x$. Consequently, combining \eqref{ZW002}--\eqref{ZW005}, we deduce that for every $i>N$,
\begin{align*}
\left|\Phi_{i}(x,R)\right|\leq C(p,n,\sigma,\alpha,\mathcal{M})
\begin{cases}
\varepsilon^{\min\{(1-\sigma)p,(\sigma p+\alpha)(p-2)+\alpha\}},&\text{if $\sigma p+\alpha\in(0,1]$},\\
\varepsilon^{\min\{(1-\sigma)p,p+\alpha-2\}},&\text{if $\sigma p+\alpha\in(1,\infty)$},
\end{cases}
\end{align*}
which implies that
\begin{align}\label{MZ001}
\lim_{i\rightarrow\infty}\Phi_{i}(x,R)=0.
\end{align}

Note that $\{(-\Delta)^{\sigma}_{p}u_{i}\}$ is a pointwise convergent sequence, we then deduce from \eqref{MZ000} and \eqref{MZ001} that
\begin{align}\label{MZ002}
\lim\limits_{i\rightarrow\infty}\Psi_{i}(x,R)\;\,\text{exists and is finite}.
\end{align}
Since $u\in \mathcal{L}_{\sigma p}(\mathbb{R}^{n})$ and $R>>|x|+1$, then
\begin{align*}
&\limsup_{R\rightarrow\infty}\int_{B_{R}^{c}(0)}\frac{|u(x)-u(y)|^{p-1}}{|x-y|^{n+\sigma p}}dy\notag\\
\leq &\limsup_{R\rightarrow\infty}\left(\frac{R}{R-|x|}\right)^{n+\sigma p}\int_{B_{R}^{c}(0)}\frac{C(p)(u^{p-1}(x)+u^{p-1}(y))}{|y|^{n+\sigma p}}dy=0,
\end{align*}
which yields that
\begin{align*}
\lim_{R\rightarrow\infty}\mathfrak{G}(x,R)=0.
\end{align*}
This, in combination with \eqref{MZ000} and \eqref{MZ001}--\eqref{MZ002}, leads to that $\lim\limits_{R\rightarrow\infty}\lim\limits_{i\rightarrow\infty}\Psi_{i}(x,R)$ exists and is finite,
\begin{align}\label{MZ003}
(-\Delta)^{\sigma}_{p}u(x)-\lim_{i\rightarrow\infty}(-\Delta)^{\sigma}_{p}u_{i}(x)=\lim_{R\rightarrow\infty}\lim_{i\rightarrow\infty}\Psi_{i}(x,R).
\end{align}

Denote
\begin{align*}
\mathcal{K}_{1}:=&-u_{i}^{p-2}(y)u_{i}(x),\notag\\
\mathcal{K}_{2}:=&\left(u_{i}^{p-2}(y)-|u_{i}(x)-u_{i}(y)|^{p-2}\right)u_{i}(x),\notag\\
\mathcal{K}_{3}:=&-\left(u_{i}^{p-2}(y)-|u_{i}(x)-u_{i}(y)|^{p-2}\right)u_{i}(y),\notag\\
\Theta:=&-|u_{i}(x)-u_{i}(y)|^{p-2}(u_{i}(x)-u_{i}(y)).
\end{align*}
Then we have
\begin{align}\label{MZ005}
u_{i}^{p-1}(y)-\sum^{3}_{j=1}|\mathcal{K}_{j}|\leq\Theta=u_{i}^{p-1}(y)+\sum^{3}_{j=1}\mathcal{K}_{j}\leq u_{i}^{p-1}(y)+\sum^{3}_{j=2}|\mathcal{K}_{j}|.
\end{align}

For any given $\varepsilon>0$, it follows from Young's inequality that
\begin{align}\label{MZ006}
|\mathcal{K}_{1}|\leq \varepsilon u^{p-1}_{i}(y)+\frac{C(p)}{\varepsilon^{p-2}}u^{p-1}_{i}(x).
\end{align}
We now divide into three cases to estimate $\mathcal{K}_{2}$ and $\mathcal{K}_{3}$ in the following.

{\bf Case 1.} Consider $2<p\leq3$. Since
\begin{align*}
&u_{i}^{p-2}(y)\leq (|u_{i}(y)-u_{i}(x)|+u_{i}(x))^{p-2}\leq|u_{i}(y)-u_{i}(x)|^{p-2}+u_{i}^{p-2}(x),\notag\\
&|u_{i}(y)-u_{i}(x)|^{p-2}\leq u_{i}^{p-2}(y)+u_{i}^{p-2}(x),
\end{align*}
then
\begin{align*}
\left|u_{i}^{p-2}(y)-|u_{i}(x)-u_{i}(y)|^{p-2}\right|\leq u_{i}^{p-2}(x).
\end{align*}
Hence it follows from Young's inequality that
\begin{align}\label{MZ007}
|\mathcal{K}_{2}|\leq u^{p-1}_{i}(x),\quad|\mathcal{K}_{3}|\leq \varepsilon u_{i}^{p-1}(y)+\frac{C(p)}{\varepsilon^{\frac{1}{p-2}}}u^{p-1}_{i}(x).
\end{align}
Substituting \eqref{MZ006}--\eqref{MZ007} into \eqref{MZ005}, we derive
\begin{align}\label{MZ008}
(1-2\varepsilon)u^{p-1}_{i}(y)-\frac{C(p)}{\varepsilon^{p-2}}u^{p-1}_{i}(x)\leq\Theta\leq(1+\varepsilon)u^{p-1}_{i}(y)+\frac{C(p)}{\varepsilon^{\frac{1}{p-2}}}u^{p-1}_{i}(x).
\end{align}

{\bf Case 2.} Consider the case when $p>3$ is an integer. From the binomial theorem and Young's inequality, we have
\begin{align}\label{MZ009}
(a+b)^{p-2}=&a^{p-2}+\sum^{p-2}_{j=1}C_{p-2}^{j}a^{p-2-j}b^{j}\leq(1+\varepsilon)a^{p-2}+C(p)b^{p-2}\sum^{p-2}_{j=1}\varepsilon^{-\frac{p-2-k}{k}}\notag\\ \leq&(1+\varepsilon)a^{p-2}+\frac{C(p)}{\varepsilon^{p-3}}b^{p-2},\quad \text{for any }a,b\geq0.
\end{align}
Using \eqref{MZ009}, we deduce
\begin{align*}
u_{i}^{p-2}(y)\leq(|u_{i}(y)-u_{i}(x)|+u_{i}(x))^{p-2}\leq(1+\varepsilon)|u_{i}(x)-u_{i}(y)|^{p-2}+\frac{C(p)}{\varepsilon^{p-3}}u_{i}^{p-2}(x),
\end{align*}
which implies that
\begin{align*}
u_{i}^{p-2}(y)-|u_{i}(y)-u_{i}(x)|^{p-2}\leq&\varepsilon|u_{i}(x)-u_{i}(y)|^{p-2}+\frac{C(p)}{\varepsilon^{p-3}}u_{i}^{p-2}(x)\notag\\
\leq&\varepsilon(1+\varepsilon)u^{p-2}_{i}(y)+\frac{C(p)}{\varepsilon^{p-3}}u^{p-2}_{i}(x).
\end{align*}
Analogously,
\begin{align*}
|u_{i}(y)-u_{i}(x)|^{p-2}-u_{i}^{p-2}(y)\leq\varepsilon u^{p-2}_{i}(y)+\frac{C(p)}{\varepsilon^{p-3}}u^{p-2}_{i}(x).
\end{align*}
Hence, we have
\begin{align}\label{MZ010}
\left|u_{i}^{p-2}(y)-|u_{i}(y)-u_{i}(x)|^{p-2}\right|\leq\varepsilon(1+\varepsilon)u^{p-2}_{i}(y)+\frac{C(p)}{\varepsilon^{p-3}}u^{p-2}_{i}(x).
\end{align}
Utilizing \eqref{MZ010} and Young's inequality, we obtain
\begin{align*}
|\mathcal{K}_{2}|\leq&\varepsilon(1+\varepsilon)u^{p-2}_{i}(y)u_{i}(x)+\frac{C(p)}{\varepsilon^{p-3}}u^{p-1}_{i}(x)\leq\varepsilon u^{p-1}_{i}(y)+\frac{C(p)}{\varepsilon^{p-3}}u^{p-1}_{i}(x),\notag\\
|\mathcal{K}_{3}|\leq&\varepsilon(1+\varepsilon)u^{p-1}_{i}(y)+\frac{C(p)}{\varepsilon^{p-3}}u^{p-2}_{i}(x)u_{i}(y)\leq\varepsilon(2+\varepsilon) u^{p-1}_{i}(y)+\frac{C(p)}{\varepsilon^{p-2}}u^{p-1}_{i}(x),
\end{align*}
which, in combination with \eqref{MZ005}--\eqref{MZ006}, gives that
\begin{align}
\Theta\leq&(1+3\varepsilon+\varepsilon^{2})u^{p-1}_{i}(y)+\frac{C(p)}{\varepsilon^{p-2}}u^{p-1}_{i}(x),\label{MZ011}\\
\Theta\geq&(1-4\varepsilon-\varepsilon^{2})u^{p-1}_{i}(y)-\frac{C(p)}{\varepsilon^{p-2}}u^{p-1}_{i}(x).\label{MZ012}
\end{align}

{\bf Case 3.} Consider the case when $p>3$ is not an integer. On one hand, making use of \eqref{MZ009}, we obtain
\begin{align}\label{MZ013}
|u_{i}(x)-u_{i}(y)|^{p-2}\leq&(u_{i}(x)+u_{i}(y))^{[p-2]+(p-[p])}\notag\\
\leq&\left((1+\varepsilon)u^{[p-2]}_{i}(y)+\frac{C(p)}{\varepsilon^{[p-3]}}u^{[p-2]}_{i}(x)\right)\left(u^{p-[p]}_{i}(y)+u_{i}^{p-[p]}(x)\right)\notag\\
=&(1+\varepsilon)u^{p-2}_{i}(y)+\frac{C(p)}{\varepsilon^{[p-3]}}u^{p-[p]}_{i}(y)u^{[p-2]}_{i}(x)\notag\\
&+(1+\varepsilon)u^{[p-2]}_{i}(y)u^{p-[p]}_{i}(x)+\frac{C(p)}{\varepsilon^{[p-3]}}u^{p-2}_{i}(x).
\end{align}
From Young's inequality, we deduce
\begin{align}
&\frac{C(p)}{\varepsilon^{[p-3]}}u^{p-[p]}_{i}(y)u^{[p-2]}_{i}(x)\leq\varepsilon u^{p-2}_{i}(y)+\frac{C(p)}{\varepsilon^{p-3}}u^{p-2}_{i}(x),\label{MZ015}\\
&(1+\varepsilon)u^{[p-2]}_{i}(y)u^{p-[p]}_{i}(x)\leq\varepsilon u^{p-2}_{i}(y)+\frac{C(p)}{\varepsilon^{\frac{[p-2]}{p-[p]}}}u^{p-2}_{i}(x).\label{MZ016}
\end{align}
Substituting \eqref{MZ015}--\eqref{MZ016} into \eqref{MZ013}, it follows that
\begin{align}\label{MZ017}
|u_{i}(x)-u_{i}(y)|^{p-2}-u^{p-2}_{i}(y)\leq3\varepsilon u^{p-2}_{i}(y)+\frac{C(p)}{\varepsilon^{\frac{[p-2]}{p-[p]}}}u^{p-2}_{i}(x).
\end{align}

On the other hand, using \eqref{MZ009} again, we have
\begin{align}\label{MZ018}
u_{i}^{p-2}(y)\leq&(|u_{i}(y)-u_{i}(x)|+u_{i}(x))^{[p-2]+(p-[p])}\notag\\
\leq&\left((1+\varepsilon)|u_{i}(x)-u_{i}(y)|^{[p-2]}+\frac{C(p)}{\varepsilon^{[p-3]}}u^{[p-2]}_{i}(x)\right)\notag\\
&\cdot\left(|u_{i}(x)-u_{i}(y)|^{p-[p]}+u^{p-[p]}_{i}(x)\right)\notag\\
=&(1+\varepsilon)|u_{i}(x)-u_{i}(y)|^{p-2}+\frac{C(p)}{\varepsilon^{[p-3]}}u^{[p-2]}_{i}(x)|u_{i}(x)-u_{i}(y)|^{p-[p]}\notag\\
&+(1+\varepsilon)u^{p-[p]}_{i}(x)|u_{i}(x)-u_{i}(y)|^{[p-2]}+\frac{C(p)}{\varepsilon^{[p-3]}}u^{p-2}_{i}(x).
\end{align}
It follows from Young's inequality that
\begin{align}
&\frac{C(p)}{\varepsilon^{[p-3]}}u^{[p-2]}_{i}(x)|u_{i}(x)-u_{i}(y)|^{p-[p]}\leq\varepsilon|u_{i}(x)-u_{i}(y)|^{p-2}+\frac{C(p)}{\varepsilon^{p-3}}u^{p-2}_{i}(x),\label{MZ019}\\
&(1+\varepsilon)u^{p-[p]}_{i}(x)|u_{i}(x)-u_{i}(y)|^{[p-2]}\leq\varepsilon|u_{i}(x)-u_{i}(y)|^{p-2}+\frac{C(p)}{\varepsilon^{\frac{[p-2]}{p-[p]}}}u^{p-2}_{i}(x).\label{MZ020}
\end{align}
Combining \eqref{MZ017}--\eqref{MZ020}, we deduce
\begin{align*}
u_{i}^{p-2}(y)-|u_{i}(x)-u_{i}(y)|^{p-2}\leq&3\varepsilon|u_{i}(x)-u_{i}(y)|^{p-2}+\frac{C(p)}{\varepsilon^{\frac{[p-2]}{p-[p]}}}u^{p-2}_{i}(x)\notag\\
\leq&3\varepsilon(1+3\varepsilon)u^{p-2}_{i}(y)+\frac{C(p)}{\varepsilon^{\frac{[p-2]}{p-[p]}}}u^{p-2}_{i}(x).
\end{align*}
This, together with \eqref{MZ017} again, gives that
\begin{align}\label{MZ021}
\left|u_{i}^{p-2}(y)-|u_{i}(x)-u_{i}(y)|^{p-2}\right|\leq&3\varepsilon(1+3\varepsilon)u^{p-2}_{i}(y)+\frac{C(p)}{\varepsilon^{\frac{[p-2]}{p-[p]}}}u^{p-2}_{i}(x).
\end{align}
In light of \eqref{MZ021}, it follows from Young's inequality that
\begin{align}
|\mathcal{K}_{2}|\leq&3\varepsilon(1+3\varepsilon)u^{p-2}_{i}(y)u_{i}(x)+\frac{C(p)}{\varepsilon^{\frac{[p-2]}{p-[p]}}}u^{p-1}_{i}(x)\notag\\
\leq&\varepsilon u^{p-1}_{i}(y)+\frac{C(p)}{\varepsilon^{\frac{[p-2]}{p-[p]}}}u^{p-1}_{i}(x),\label{MZ022}\\
|\mathcal{K}_{3}|\leq&3\varepsilon(1+3\varepsilon)u^{p-1}_{i}(y)+\frac{C(p)}{\varepsilon^{\frac{[p-2]}{p-[p]}}}u^{p-2}_{i}(x)u_{i}(y)\notag\\
\leq&\varepsilon(4+9\varepsilon) u^{p-1}_{i}(y)+\frac{C(p)}{\varepsilon^{\frac{[p-1]}{p-[p]}}}u^{p-1}_{i}(x).\label{MZ023}
\end{align}
Therefore, substituting \eqref{MZ006} and \eqref{MZ022}--\eqref{MZ023} into \eqref{MZ005}, we derive
\begin{align}
\Theta\leq&(1+5\varepsilon+9\varepsilon^{2})u^{p-1}_{i}(y)+\frac{C(p)}{\varepsilon^{\frac{[p-1]}{p-[p]}}}u^{p-1}_{i}(x),\label{MZ026}\\
\Theta\geq&(1-6\varepsilon-9\varepsilon^{2})u^{p-1}_{i}(y)-\frac{C(p)}{\varepsilon^{\frac{[p-1]}{p-[p]}}}u^{p-1}_{i}(x).\label{MZ027}
\end{align}

Observe that
\begin{align}\label{MZ028}
\lim_{R\rightarrow\infty}\lim_{i\rightarrow\infty}\int_{B_{R}^{c}(0)}\frac{u_{i}^{p-1}(x)}{|x-y|^{n+\sigma p}}dy=&u^{p-1}(x)\lim_{R\rightarrow\infty}\int_{B_{R}^{c}(0)}\frac{dy}{|x-y|^{n+\sigma p}}\notag\\
\leq&u^{p-1}(x)\lim_{R\rightarrow\infty}\int_{B_{R-|x|}^{c}(x)}\frac{dy}{|x-y|^{n+\sigma p}}=0.
\end{align}
Since $\lim\limits_{R\rightarrow\infty}\lim\limits_{i\rightarrow\infty}\Psi_{i}(x,R)$ exists and is finite, it follows from \eqref{MZ008}, \eqref{MZ011}--\eqref{MZ012} and \eqref{MZ026}--\eqref{MZ028} that
\begin{align*}
\lim_{R\rightarrow\infty}\lim_{i\rightarrow\infty}\Psi_{i}(x,R)\leq&c_{n,\sigma p}(1+\varepsilon_{p}^{(1)})\liminf_{R\rightarrow\infty}\liminf_{i\rightarrow\infty}\int_{B^{c}_{R}(0)}\frac{u^{p-1}_{i}(y)}{|x-y|^{n+\sigma p}}dy,\\
\lim_{R\rightarrow\infty}\lim_{i\rightarrow\infty}\Psi_{i}(x,R)\geq&c_{n,\sigma p}(1-\varepsilon_{p}^{(2)})\limsup_{R\rightarrow\infty}\limsup_{i\rightarrow\infty}\int_{B^{c}_{R}(0)}\frac{u^{p-1}_{i}(y)}{|x-y|^{n+\sigma p}}dy,
\end{align*}
where
\begin{align*}
\varepsilon_{p}^{(1)}=&
\begin{cases}
\varepsilon,&\text{if }\,2<p\leq3,\\
\varepsilon(3+\varepsilon),&\text{if }\,p>3\;\text{is an integer},\\
\varepsilon(5+9\varepsilon),&\text{if }\,p>3\;\text{is not an integer},
\end{cases}\\
\varepsilon^{(2)}_{p}=&
\begin{cases}
2\varepsilon,&\text{if }\,2<p\leq3,\\
\varepsilon(4+\varepsilon),&\text{if }\,p>3\;\text{is an integer},\\
3\varepsilon(2+3\varepsilon),&\text{if }\,p>3\;\text{is not an integer}.
\end{cases}
\end{align*}
Due to the fact that $R>>|x|+1$, we have
\begin{align*}
\frac{(R-|x|)|y|}{R}\leq|y-x|\leq\frac{(R+|x|)|y|}{R},\quad\text{for }y\in B_{R}^{c}(0).
\end{align*}
Hence, we deduce
\begin{align*}
\lim_{R\rightarrow\infty}\lim_{i\rightarrow\infty}\Psi_{i}(x,R)\leq&c_{n,\sigma p}(1+\varepsilon_{p}^{(1)})\liminf_{R\rightarrow\infty}\liminf_{i\rightarrow\infty}\int_{B^{c}_{R}(0)}\frac{u^{p-1}_{i}(y)}{|y|^{n+\sigma p}}dy,\\
\lim_{R\rightarrow\infty}\lim_{i\rightarrow\infty}\Psi_{i}(x,R)\geq&c_{n,\sigma p}(1-\varepsilon_{p}^{(2)})\limsup_{R\rightarrow\infty}\limsup_{i\rightarrow\infty}\int_{B^{c}_{R}(0)}\frac{u^{p-1}_{i}(y)}{|y|^{n+\sigma p}}dy,
\end{align*}
By virtue of the arbitrariness of $\varepsilon$ and $\{u_{i}\}$ is nonnegative, we obtain
\begin{align*}
\lim_{R\rightarrow\infty}\lim_{i\rightarrow\infty}\Psi_{i}(x,R)=c_{n,\sigma p}\lim_{R\rightarrow\infty}\lim_{i\rightarrow\infty}\int_{B^{c}_{R}(0)}\frac{u^{p-1}_{i}(y)}{|y|^{n+\sigma p}}dy\geq0.
\end{align*}
This, together with \eqref{MZ003}, yields that Theorem \ref{thm001} holds.

\end{proof}

In order to show that the limit constant $\theta$ captured in Theorem \ref{thm001} may be positive, we consider a sequence of nonnegative functions in the following. Choose a smooth cut-off function $\eta$ satisfying that
\begin{align}\label{eta}
\text{$\eta(t)\equiv0$ in $(-\infty,0]$, $\eta(t)\equiv1$ in $[1,\infty)$, and $0\leq\eta(t)\leq1$ in $[0,1]$}.
\end{align}
Then for any $0<s<t$ and $j\geq1$, define
\begin{align}\label{exa001}
v_{j}(x):=j^{-s}w_{j}(R_{j}^{-1}x),\quad w_{j}(x):=
\begin{cases}
j^{s}+j^{t}\phi(x),\quad\mathrm{in}\;B_{6},\\
(1-\psi(x))(j^{s}+j^{t}),\quad\mathrm{in}\; B_{6}^{c},
\end{cases}
\end{align}
where $\phi(x)=\eta(|x|-3)$, and $\psi(x)=\eta(|x|-6)$, $R_{j}=j^{\frac{(t-s)(p-1)}{\sigma p}}\beta^{\frac{1}{\sigma p}}$ with
\begin{align}\label{MZ031}
\beta:=c_{n,\sigma p}\left(\int_{B_{4}\setminus B_{3}}\frac{\phi^{p-1}(y)}{|y|^{n+\sigma p}}dy+\int_{B_{6}\setminus B_{4}}\frac{dy}{|y|^{n+\sigma p}}+\int_{B_{6}^{c}}\frac{(1-\psi(y))^{p-1}}{|y|^{n+\sigma p}}dy\right).
\end{align}

\begin{example}\label{thm-exm002}
Let $n\geq1$, $p>2$ and $0<\sigma<1$. If condition \eqref{exa001} holds, then we obtain that $v_{j}$ converges to $1$ in $C^{2}_{loc}(\mathbb{R}^{n})$, and
\begin{align*}
\lim_{i\rightarrow\infty}(-\Delta)^{\sigma}_{p}v_{j}(x)=-1.
\end{align*}

\end{example}
\begin{remark}
We here would like to point out that the examples constructed in Example \ref{thm-exm002} and Theorem \ref{thmm003} were first given in \cite{DJXY2021}.

\end{remark}

\begin{proof}
It is easily seen from \eqref{exa001} that $v_{j}\in C_{c}^{\infty}(\mathbb{R}^{n})$, $v_{j}\geq0$ in $\mathbb{R}^{n}$, $v_{j}=1$ in $B_{R_{j}}$, and $\|v_{j}-1\|_{C^{2}_{loc}}\rightarrow0$, as $i\rightarrow\infty$. A direct computation gives that
\begin{align}\label{MZ032}
(-\Delta)^{\sigma}_{p}v_{j}(x)=j^{-s(p-1)}R_{j}^{-\sigma p}(-\Delta)^{\sigma}_{p}w_{j}(R_{j}^{-1}x),\quad\text{for $x\in B_{R_{j}}$}.
\end{align}
For any fixed $x\in\mathbb{R}^{n}$, we have
\begin{align*}
&j^{-t(p-1)}(-\Delta)^{\sigma}_{p}w_{j}(R_{j}^{-1}x)\notag\\
&=-c_{n,\sigma p}\int_{B_{4}\setminus B_{3}}\frac{\phi^{p-1}(y)}{|R_{j}^{-1}x-y|^{n+\sigma p}}dy-c_{n,\sigma p}\int_{B_{6}\setminus B_{4}}\frac{dy}{|R_{j}^{-1}x-y|^{n+\sigma p}}\notag\\
&\quad+\frac{c_{n,\sigma p}}{j^{t-s}}\int_{B_{6}^{c}}\frac{\psi(y)|\psi(y)-1+\psi(y)j^{-(t-s)}|^{p-2}}{|R_{j}^{-1}x-y|^{n+\sigma p}}dy\notag\\
&\quad-c_{n,\sigma p}\int_{B_{6}^{c}}\frac{(1-\psi(y))|\psi(y)-1+\psi(y)j^{-(t-s)}|^{p-2}}{|R_{j}^{-1}x-y|^{n+\sigma p}}dy\notag\\
&\quad\rightarrow-\beta,\quad\text{as $j$ goes to $\infty$,}
\end{align*}
where $\beta$ is defined by \eqref{MZ031}. This, together with \eqref{MZ032}, gives that
\begin{align*}
\lim_{j\rightarrow\infty}(-\Delta)^{\sigma}_{p}v_{j}(x)=\beta^{-1}\lim_{j\rightarrow\infty}j^{-t(p-1)}(-\Delta)^{\sigma}_{p}w_{j}(R_{j}^{-1}x)=-1,\quad\text{in $\mathbb{R}^{n}$}.
\end{align*}
The proof is complete.

\end{proof}

\section{Blow-up analysis for the extended fractional Nirenberg problem}
The extended fractional Nirenberg problem is equivalent to investigating the following equation
\begin{align}\label{ZWM001}
(-\Delta)^{\sigma}_{p}u(x)=K(x)u^{q(p-1)}(x),\quad\text{for $x\in\mathbb{R}^{n}$,}
\end{align}
where $p\geq2$ and $q\in\mathbb{R}$. It has been shown in \cite{DJXY2021} that there arises blow-up phenomena for the linear fractional Laplacian due to the nonzero constant $\theta$ captured in Theorem \ref{thm001}. Specially, for $p=2$, the compactness of solutions to \eqref{ZWM001} will fail in the region where $K$ is negative. In the following, we follow the proof of Theorem 1.3 in \cite{DJXY2021} and extend the result to the nonlinear case of $p>2$. On the other hand, when $K$ is positive, Jin, Li and Xiong \cite{JLX2014,JLX2015,JLX2017} derived a priori estimates for the fractional equation \eqref{ZWM001} with $p=2$.

While these above-mentioned works are related to the fractional Nirenberg problem, there is another direction of research to study the classical elliptic equation $-\Delta u=K(x)u^{p}$. When $n=1,2$ and $1<p<\infty$, or $n\geq3$ and $1<p<\frac{n+2}{n-2}$, $p$ is called a subcritical Sobolev exponent, while it is the critical Sobolev exponent if $n\geq3$ and $p=\frac{n+2}{n-2}$. In particular, the elliptic equation in the case of critical Sobolev exponent corresponds to the Nirenberg problem, which is to seek a new metric conformal to the flat metric on $\mathbb{R}^{n}$ so that its scalar curvature is $K(x)$. Generally, it needs to establish priori estimates of the solutions for the purpose of obtaining the existence of solutions. We refer to \cite{GS1981,GS198102} for the subcritical case. With regard to the critical case, see \cite{CGY1993,L1995,SZ1996} for positive functions $K$ and \cite{CL1997,L1998} for $K$ changing signs, respectively.

\begin{theorem}\label{thmm003}
Assume that $n\geq1$, $p>2$, $0<\sigma<1$, $q\in\mathbb{R}$ and $s>-\frac{\sigma p}{p-1}$. Then there exist two positive constants $c_{0}=c_{0}(n,\sigma,p,q,s)$ and $C_{0}=C_{0}(n,\sigma,p,q,s)$, a sequence of functions $\{K_{j}\}\subset C^{\infty}(\mathbb{R}^{n})$ satisfying
\begin{align*}
-C_{0}\leq K_{j}(x)\leq-c_{0},\;\,c_{0}\leq|\nabla K_{j}(x)|\leq C_{0},\;\text{and}\;|\nabla^{2}K_{j}(x)|\leq C_{0},\quad\text{in $B_{2}$},
\end{align*}
and a sequence of positive functions $\{u_{j}\}\subset C^{\infty}(\mathbb{R}^{n})$ such that
\begin{align*}
(-\Delta)^{\sigma}_{p}u_{j}(x)=K_{j}(x)u_{j}^{q(p-1)}(x),\;\,\text{for } x\in\mathbb{R}^{n},\quad|x|^{s}u_{j}(x)\rightarrow1,\quad\text{as $|x|\rightarrow\infty$},
\end{align*}
and
\begin{align*}
\min\limits_{\overline{B}_{1}}u_{j}\rightarrow\infty,\quad\text{as $j\rightarrow\infty$}.
\end{align*}

\end{theorem}

\begin{proof}
Let $\eta$ and $\phi$ be defined in \eqref{eta} and \eqref{exa001}. For $q\in\mathbb{R}$ and $s>-\frac{\sigma p}{p-1}$, let
\begin{align*}
u_{j}(x)=&
\begin{cases}
j+j^{q}\phi(x),&\quad\mathrm{in}\;B_{R},\\
(1-\varphi(x))(j+j^{q})+\varphi(x)|x|^{-s},&\quad\mathrm{in}\;B_{R}^{c},
\end{cases}
\end{align*}
where $\varphi(x)=\eta(|x|-R)$ and $R=R(n,p,q,\sigma,s,j)>9$ is a sufficiently large constant to be determined later. Then $u_{j}\in C^{\infty}(\mathbb{R}^{n})\cap \mathcal{L}_{\sigma p}(\mathbb{R}^{n})$ and $u_{j}>0$ in $\mathbb{R}^{n}$. Denote
\begin{align*}
K_{j}(x):=\frac{(-\Delta)^{\sigma}_{p}u_{j}(x)}{u_{j}^{q(p-1)}(x)},\quad\text{in $\mathbb{R}^{n}$.}
\end{align*}
Then $K_{j}\in C^{\infty}(\mathbb{R}^{n})$. Moreover, $\{K_{j}\}$ satisfies the following properties: there exists four positive constants $C_{i}:=C_{i}(n,\sigma,p)$, $i=1,2,3,4$, such that for every $j\geq1$,
\begin{itemize}
{\it
\item[(\bf{K1})] $-C_{1}\leq K_{j}(x)\leq-C_{2}$, and $\sum\limits^{3}_{i=1}|\nabla^{i}K_{j}(x)|\leq C_{3}$ in $B_{2}$;
\item[(\bf{K2})] $\nabla^{2}K_{j}(0)\leq-C_{4}\mathbf{I}_{n}$, where $\mathbf{I}_{n}$ denotes $n\times n$ identity matrix.}
\end{itemize}
We first prove $(\mathbf{K1})$. Observe that
\begin{align}\label{MZ036}
c_{n,\sigma p}^{-1}K_{j}(x)=&-\int_{B_{4}\setminus B_{3}}\frac{\phi^{p-1}(y)}{|x-y|^{n+\sigma p}}dy-\int_{B_{R}\setminus B_{4}}\frac{dy}{|x-y|^{n+\sigma p}}\notag\\
&+\int_{B_{R}^{c}}\frac{|\mathcal{A}_{\varphi}(y)|^{p-2}\mathcal{A}_{\varphi}(y)}{|x-y|^{n+\sigma p}}dy:=\sum^{3}_{i=1}J_{i},
\end{align}
where $\mathcal{A}_{\varphi}(y):=\varphi(y)-1+j^{1-q}\varphi(y)-j^{-q}|y|^{-s}\varphi(y)$. For simplicity, let
\begin{align*}
\gamma:=&\gamma(n,\sigma,p)=\int_{B_{1}^{c}}\frac{dy}{|y|^{n+\sigma p}}=\frac{|\mathbb{S}^{n-1}|}{\sigma p},\notag\\ \tau:=&\tau(n,\sigma,p,s)=\int_{B_{1}^{c}}\frac{dy}{|y|^{n+\sigma p+s(p-1)}}=\frac{|\mathbb{S}^{n-1}|}{\sigma p+s(p-1)}.
\end{align*}
A straightforward computation yields that
\begin{align*}
0\geq J_{1}\geq-\int_{B_{1}(x)^{c}}\frac{dy}{|x-y|^{n+\sigma p}}=-\gamma,
\end{align*}
and
\begin{align*}
-\gamma\leq J_{2}\leq-\int_{B_{R-2}\setminus B_{6}}\frac{dy}{|y|^{n+\sigma p}}=-(6^{-\sigma p}-(R-2)^{-\sigma p})\gamma.
\end{align*}
For $x\in B_{2}$, $y\in B_{R}^{c}$, we have $|x-y|\geq|y|/2$ in virtue of $R>9$. Then
\begin{align*}
|J_{3}|\leq&2^{(\sigma+1)p+n-2}\int_{B_{R}^{c}}\frac{(1+j^{1-q})^{p-1}+j^{-q(p-1)}|y|^{-s(p-1)}}{|y|^{n+\sigma p}}dy\notag\\
=&2^{(\sigma+1)p+n-2}\gamma(1+j^{1-q})^{p-1}R^{-\sigma p}+2^{(\sigma+1)p+n-2}\tau j^{-q(p-1)}R^{-\sigma p-s(p-1)}.
\end{align*}
For a sufficiently large $R>9$, we have
\begin{align*}
(R-2)^{-\sigma p}\gamma+2^{(\sigma+1)p+n-2}R^{-\sigma p}\left(\gamma(1+j^{1-q})^{p-1}+\tau j^{-q(p-1)}R^{-s(p-1)}\right)\leq\frac{\gamma6^{-\sigma p}}{2},
\end{align*}
which implies that
\begin{align*}
-3c_{n,\sigma p}\gamma\leq K_{j}(x)\leq-\frac{c_{n,\sigma p}\gamma6^{-\sigma p}}{2},\quad\forall\;x\in B_{2},\;j\geq1.
\end{align*}
Furthermore, after differentiating \eqref{MZ036}, it follows from a similar calculation that
\begin{align*}
\sum^{3}_{i=1}|\nabla^{i}K_{j}(x)|\leq C(n,\sigma,p),\quad\text{for $x\in B_{2},\;j\geq1$.}
\end{align*}

We proceed to verify property $(\bf{K2})$. A simple calculation shows that for $y\in B_{3}^{c}$,
\begin{align}\label{MZ039}
\partial_{x_{k}x_{l}}^{2}\left(\frac{1}{|x-y|^{n+\sigma p}}\right)(0)=\frac{(n+\sigma p)[(n+\sigma p+2)y_{k}y_{l}-\delta_{kl}|y|^{2}]}{|y|^{n+\sigma p+4}}.
\end{align}
Since the integral domain is symmetric, then we see from \eqref{MZ036}--\eqref{MZ039} that
\begin{align*}
\partial^{2}_{x_{k}x_{l}}K_{j}(0)=0,\quad\text{for $k\neq l$.}
\end{align*}
If $k=l$, it follows from the radial symmetry of $\phi$ and $\varphi$ that
\begin{align*}
&[(n+\sigma p)c_{n,\sigma p}]^{-1}\partial^{2}_{x_{k}x_{k}}K_{j}(0)\notag\\
&=-\frac{\sigma p+2}{n}\left(\int_{B_{4}\setminus B_{3}}\frac{\phi^{p-1}(y)}{|y|^{n+\sigma p+2}}dy+\int_{B_{R}^{c}}\frac{|\mathcal{A}_{\varphi}(y)|^{p-2}(j^{1-q}\varphi(y)-\mathcal{A}_{\varphi}(y))}{|y|^{n+\sigma p+2}}dy\right)\notag\\
&\quad-\frac{\sigma p+2}{n}\left(\int_{B_{R}\setminus B_{4}}\frac{dy}{|y|^{n+\sigma p+2}}-j^{1-q}\int_{B_{R}^{c}}\frac{|\mathcal{A}_{\varphi}(y)|^{p-2}\varphi(y)}{|y|^{n+\sigma p+2}}dy\right)\notag\\
&\leq-|B_{1}|\left(4^{-(\sigma p+2)}-R^{-(\sigma p+2)}-3j^{1-q}R^{-(\sigma p+2)}\right)\notag\\
&\leq-|B_{1}|4^{-(\sigma p+3)},\quad\text{for a sufficiently large $R>9$,}
\end{align*}
where we used the fact that $|\mathcal{A}_{\varphi}(y)|^{p-2}\varphi(y)\leq3$ in $B_{R}^{c}$. That is, property $(\bf{K2})$ holds.

From the radial symmetry of $u_{j}$ with respect to the origin, we know that $K_{j}$ is also radially symmetric. Then we have
\begin{align*}
\nabla K_{j}(0)=0,
\end{align*}
which, together with $(\bf{K1})$--$(\bf{K2})$, leads to that for $j\geq1$,
\begin{align}\label{MZ050}
|\nabla K_{j}(x)|\geq c_{1},\quad\text{in $B_{2\varepsilon_{0}}(4\varepsilon_{0}e_{1})$,}
\end{align}
where $e_{1}=(1,0,...,0)\in\mathbb{R}^{n}$, $\varepsilon_{0}:=\varepsilon_{0}(n,p,\sigma)\in(0,1/4)$ is a small constant and $c_{1}:=c_{1}(n,p,\sigma)$ is a positive constant.

Define
\begin{align*}
\bar{u}_{j}(x):=\varepsilon_{0}^{s}u_{j}(\varepsilon_{0}(x+4e_{1})),\quad\mathrm{and}\;\,\bar{K}_{j}(x):=\varepsilon_{0}^{\sigma p-s(p-1)(q-1)}K_{j}(\varepsilon_{0}(x+4e_{1})).
\end{align*}
Therefore,
\begin{align*}
(-\Delta)^{\sigma}_{p}\bar{u}_{j}=\bar{K}_{j}(x)\bar{u}_{j}^{q(p-1)},\quad\text{for $x\in\mathbb{R}^{n}$}.
\end{align*}
Then combining $(\bf{K1})$ and \eqref{MZ050}, we obtain
\begin{align*}
-\bar{C}\leq \bar{K}_{j}(x)\leq-\bar{c},\;\,\bar{c}\leq|\nabla \bar{K}_{j}(x)|\leq \bar{C},\;\text{and}\;|\nabla^{2}\bar{K}_{j}(x)|\leq \bar{C},\quad\text{in $B_{2}$},
\end{align*}
where $\bar{c}=\bar{c}(n,\sigma,p,q,s)$ and $\bar{C}=\bar{C}(n,\sigma,p,q,s)$. Moreover, recalling the definition of $u_{j}$, we have
\begin{align*}
\lim_{|x|\rightarrow\infty}|x|^{s}\bar{u}_{j}=1,\quad\text{and}\;\,\min_{\overline{B}_{1}}\bar{u}_{j}=\varepsilon_{0}^{s}j\rightarrow\infty,\quad\text{as $j\rightarrow\infty$.}
\end{align*}
The proof is finished.

\end{proof}

\noindent{\bf{\large Acknowledgements.}}
The author would like to thank Prof. C.X. Miao for his constant encouragement and useful discussions. The author was partially supported
by CPSF (2021M700358).

\bibliographystyle{plain}

\def\cprime{$'$}

\end{document}